\documentclass{amsart}

\numberwithin{equation}{section}

\usepackage{color,graphics,epsfig}

\newcommand{\calL}{\mathcal{L}}

\newcommand{\mC}{\mathbb{C}}
\newcommand{\mD}{\mathbb{D}}

\newcommand{\mF}{\mathbb{F}}

\newcommand{\mS}{\mathbb{S}}
\newcommand{\mT}{\mathbb{T}}

\newcommand{\mZ}{\mathbb{Z}}

\newcommand{\inv}{{\textrm{inv }}}

\newcommand{\nm}{\,\rule[-.6ex]{.13em}{2.3ex}\,}

\newtheorem{theorem}{Theorem}[section]
\newtheorem{lemma}[theorem]{Lemma}
\newtheorem{corollary}[theorem]{Corollary}
\newtheorem{proposition}[theorem]{Proposition}

\theoremstyle{definition}

\theoremstyle{definition}
\newtheorem{definition}[theorem]{Definition}

\theoremstyle{definition}

\begin{document}



\title[Extension of the $\nu$-metric]{Extension of the $\nu$-metric:
  the $H^\infty$ case}

\author{Joseph A. Ball}
\address{Department of Mathematics,
Virginia Tech.,
Blacksburg,  VA 24061,
USA.}
\email{joball@math.vt.edu}

\author{Amol J. Sasane}
\address{Department of Mathematics, Royal Institute of Technology,
    Stockholm, Sweden.}
\email{sasane@math.kth.se}

\begin{abstract}
  An abtract $\nu$-metric was introduced by Ball and Sasane, with a
  view towards extending the classical $\nu$-metric of Vinnicombe from
  the case of rational transfer functions to more general nonrational
  transfer function classes of infinite-dimensional linear control
  systems. In this short note, we give an additional concrete special
  instance of the abstract $\nu$-metric, by verifying all the
  assumptions demanded in the abstract set-up. This example links the
  abstract $\nu$-metric with the one proposed by Vinnicombe as a
  candidate for the $\nu$-metric for nonrational plants.
\end{abstract}

\maketitle

\section{Introduction}

We recall the general {\em stabilization problem} in control theory.
Suppose that $R$ is a commutative integral domain with identity
(thought of as the class of stable transfer functions) and let
$\mF(R)$ denote the field of fractions of $R$. The stabilization
problem is:

\medskip

\begin{center}
\parbox[r]{11cm}{Given $P\in (\mF(R))^{p\times
  m}$ (an unstable plant transfer function),

 find $C \in (\mF(R))^{m\times p}$ (a stabilizing controller
  transfer function),

 such that (the closed loop transfer function)
$$
H(P,C):= \left[\begin{array}{cc} P \\ I \end{array} \right]
(I-CP)^{-1} \left[\begin{array}{cc} -C & I \end{array} \right]
$$
belongs to $R^{(p+m)\times (p+m)}$ (is stable).}
\end{center}

\medskip

In the {\em robust stabilization problem}, one goes a step further.
One knows that the plant is just an approximation of reality, and so
one would really like the controller $C$ to not only stabilize the
{\em nominal} plant $P_0$, but also all sufficiently close plants $P$
to $P_0$.  The question of what one means by ``closeness'' of plants
thus arises naturally.

So one needs a function $d$ defined on
pairs of stabilizable plants such that
\begin{enumerate}
\item $d$ is a metric on the set of all stabilizable
plants,
\item $d$ is amenable to computation, and
\item stabilizability is a robust property of the plant with respect
  to this metric.
\end{enumerate}
Such a desirable metric, was introduced by Glenn Vinnicombe in
\cite{Vin} and is called the $\nu$-{\em metric}. In that paper,
essentially $R$ was taken to be the rational functions without poles
in the closed unit disk or, more generally, the disk algebra, and the
most important results were that the $\nu$-metric is indeed a metric
on the set of stabilizable plants, and moreover, one has the
inequality that if $P_0,P\in \mS(R,p,m)$, then
$$
\mu_{P,C} \geq \mu_{P_0,C}-d_{\nu}(P_0,P),
$$
where $\mu_{P,C}$ denotes the {\em stability margin} of the pair
$(P,C)$, defined by
$$
\mu_{P,C}:=\|H(P,C)\|_\infty^{-1}.
$$
This implies in particular that stabilizability is a robust property
of the plant $P$. 

The problem of what happens when $R$ is some other ring of stable
transfer functions of infinite-dimensional systems was left open in
\cite{Vin}. This problem of extending the $\nu$-metric from the
rational case to transfer function classes of infinite-dimensional
systems was addressed in \cite{BalSas}. There the starting point in
the approach was abstract. It was assumed that $R$ is any commutative
integral domain with identity which is a subset of a Banach algebra
$S$ satisfying certain assumptions, labelled (A1)-(A4), which are
recalled in Section~\ref{section_abstract_nu_metric}.  Then an
``abstract'' $\nu$-metric was defined in this setup, and it was shown
in \cite{BalSas} that it does define a metric on the class of all
stabilizable plants.  It was also shown there that stabilizability is
a robust property of the plant.

In \cite{Vin}, it was suggested that the $\nu$-metric in the case when
$R=H^\infty$ might be defined as follows. Let $P_1, P_2$ be unstable
plants with the normalized left/right coprime factorizations
\begin{eqnarray*}
P_1&=& N_{1} D_{1}^{-1}= \widetilde{D}_{1}^{-1} \widetilde{N}_{1},\\
P_2&=& N_{2} D_{2}^{-1}= \widetilde{D}_{2}^{-1} \widetilde{N}_{2},
\end{eqnarray*}
where $N_1, D_1, N_2, D_2, \widetilde{N}_{1},
\widetilde{D}_{1},\widetilde{N}_{2}, \widetilde{D}_{2}$ are matrices
with $H^\infty$ entries. Then 
\begin{equation}
\label{eq_nu_metric_GV}
d_{\nu} (P_1,P_2 )=\left\{\begin{array}{ll}
\|\widetilde{G}_2 G_1\|_\infty & \textrm{if }
 T_{G_1^* G_2} \textrm{ is Fredholm with Fredholm index } 0,\\
0& \textrm{otherwise}
\end{array}
\right.
\end{equation}
Here $\cdot^\ast$ has the usual meaning, namely: $G_1^*(\zeta)$ is the
transpose of the matrix whose entries are complex conjugates of the
entries of the matrix $G_1(\zeta)$, for $\zeta \in \mT$. Also in
the above, for a matrix $M\in (L^\infty)^{p\times m}$,
$T_M:(H^2)^m\rightarrow (H^2)^p$ denotes the {\em Toeplitz operator}
given by
$$
T_M \varphi=P_{(H^2)^p}(M\varphi) \quad (\varphi \in (H^2)^m)
$$
where $M\varphi$ is considered as an element of $(L^2)^p$ and
$P_{(H^2)^p}$ denotes the canonical orthogonal projection from
$(L^2)^p$ onto $(H^2)^p$.

Although we are unable to verify whether there is a metric $d_\nu$
such that the above holds in the case of $H^\infty$, we show that the
above does work for the somewhat smaller case when $R$ is the class
$QA$ of quasicontinuous functions analytic in the unit disk. We prove
this by showing that this case is just a special instance of the
abstract $\nu$-metric introduced in \cite{BalSas}.

The paper is organized as follows:
\begin{enumerate}
\item In Section~\ref{section_abstract_nu_metric}, we recall the
  general setup and assumptions and the abstract metric $d_\nu$ from
  \cite{BalSas}.
\item In Section~\ref{section_Hinfty}, we specialize $R$ to
  a concrete ring of stable transfer functions, and
  show that our abstract assumptions hold in this particular case.
\end{enumerate}

\section{Recap of the abstract $\nu$-metric}
\label{section_abstract_nu_metric}

\noindent We recall the setup from \cite{BalSas}:
\begin{itemize}
\item[(A1)] $R$ is commutative integral domain with identity.
\item[(A2)] $S$ is a unital commutative complex semisimple Banach
  algebra with an involution $\cdot^*$, such that $R \subset S$.  We
  use $\inv S$ to denote the invertible elements of $S$.
\item[(A3)] There exists a map $\iota: \inv S \rightarrow G$, where
  $(G,+)$ is an Abelian group with identity denoted by $\circ$, and
  $\iota$ satisfies
\begin{itemize}
\item[(I1)] $\iota(ab)= \iota (a) +\iota(b)$ ($a,b \in \inv S$).
\item[(I2)] $\iota(a^*)=-\iota(a)$ ($a\in \inv S$).
\item[(I3)] $\iota$ is locally constant, that is, $\iota$ is continuous
  when $G$ is equipped with the discrete topology.
\end{itemize}
\item[(A4)] $x\in R \cap (\inv S)$ is invertible as an element of $R$
  if and only if $\iota(x)=\circ$.
\end{itemize}

\medskip 

\noindent  We recall the following standard definitions from the factorization
approach to control theory.

\medskip 

\noindent {\bf The notation $\mF(R)$:} $\mF(R)$ denotes the field of
  fractions of $R$.

\medskip

\noindent {\bf The notation $F^*$:} If $F\in R^{p\times m}$, then $F^*\in
  S^{m\times p}$ is the matrix with the entry in the $i$th row and
  $j$th column given by $F_{ji}^*$, for all $1\leq i\leq p$, and all $
  1\leq j \leq m$.

\medskip

\noindent {\bf Right coprime/normalized coprime factorization:} Given
a matrix $P \in (\mF(R))^{p\times m}$, a factorization $P=ND^{-1}$,
where $N,D$ are matrices with entries from $R$, is called a {\em right
  coprime factorization of} $P$ if there exist matrices $X, Y$ with
entries from $R$ such that $ X N + Y D=I_m$.  If moreover it holds
that $ N^{*} N +D^{*} D =I_m$, then the right coprime factorization is
referred to as a {\em normalized} right coprime factorization of $P$.

\medskip

\noindent {\bf Left coprime/normalized coprime factorization:} 
A factorization $P=\widetilde{D}^{-1}\widetilde{N}$, where
$\widetilde{N},\widetilde{D}$ are matrices with entries from $R$, is
called a {\em left coprime factorization of} $P$ if there exist
matrices $\widetilde{X}, \widetilde{Y}$ with entries from $R$ such
that $ \widetilde{N} \widetilde{X}+\widetilde{D} \widetilde{Y}=I_p.  $
If moreover it holds that $ \widetilde{N} \widetilde{N}^{*}
+\widetilde{D}\widetilde{D}^{*}=I_p, $ then the left coprime
factorization is referred to as a {\em normalized} left coprime
factorization of $P$.  

\medskip

\noindent {\bf The notation $G, \widetilde{G}, K,\widetilde{K}$:}
Given $P \in (\mF(R))^{p\times m}$ with normalized right and left
factorizations $P=N D^{-1}$ and $P= \widetilde{D}^{-1} \widetilde{N}$,
respectively, we introduce the following matrices with entries from
$R$:
  $$
  G=\left[ \begin{array}{cc} N \\ D \end{array} \right] \quad
  \textrm{and} \quad
  \widetilde{G}=\left[ \begin{array}{cc} -\widetilde{D} &
  \widetilde{N} \end{array} \right] .
  $$
  Similarly, given $C \in (\mF(R))^{m\times p}$ with normalized right
  and left factorizations $C=N_C D_C^{-1}$ and $C=
  \widetilde{D}_C^{-1} \widetilde{N}_C$, respectively, we introduce
  the following matrices with entries from $R$:
  $$
  K=\left[ \begin{array}{cc} D_C \\ N_C \end{array} \right] \quad
  \textrm{and} \quad
  \widetilde{K}=\left[ \begin{array}{cc}
      -\widetilde{N}_C & \widetilde{D}_C \end{array} \right] .
  $$

\medskip

\noindent {\bf The notation $\mS(R,p, m)$:} We denote by $\mS(R,p, m)$ the
  set of all elements  $P\in (\mF(R))^{p\times m}$ that possess a normalized
  right coprime factorization and a normalized left coprime
  factorization.

\medskip 

We now define the metric $d_\nu$ on $\mS(R, p, m)$. But first we
specify the norm we use for matrices with entries from $S$.

\begin{definition}[$\|\cdot \|$]\label{def_sup_norm}
  Let $\mathfrak{M}$ denote the maximal ideal space of the Banach
  algebra $S$.  For a matrix $M \in S^{p\times m}$, we set
\begin{equation}
\label{norm}
\|M\|= \max_{\varphi \in \mathfrak{M}} \nm {\mathbf M}(\varphi) \nm.
\end{equation}
Here ${\mathbf M}$ denotes the entry-wise Gelfand transform of $M$,
and $\nm \cdot \nm$ denotes the induced operator norm from $\mC^{m}$
to $\mC^{p}$. For the sake of concreteness, we fix the standard
Euclidean norms on the vector spaces $\mC^{m}$ to $\mC^{p}$.
\end{definition}

The maximum in \eqref{norm} exists since $\mathfrak{M}$ is a compact
space when it is equipped with Gelfand topology, that is, the
weak-$\ast$ topology induced from $\calL( S; \mC)$. Since we have
assumed $S$ to be semisimple, the Gelfand transform
$$
\widehat{\;\cdot\;}:S \rightarrow\widehat{S}\; ( \subset
C(\mathfrak{M},\mC))
$$
is an isomorphism. If $M\in S^{1\times 1}=S$, then we note that there
are two norms available for $M$: the one as we have defined above,
namely $\|M\|$, and the norm $\|\cdot\|_S$ of $M$ as an element
of the Banach algebra $S$. But throughout this article, we will use
the norm given by \eqref{norm}.

\begin{definition}[Abstract $\nu$-metric $d_\nu$]\label{def_nu_metric}
  For $P_1, P_2 \in \mS(R,p,m)$, with the normalized left/right
  coprime factorizations
\begin{eqnarray*}
P_1&=& N_{1} D_{1}^{-1}= \widetilde{D}_{1}^{-1} \widetilde{N}_{1},\\
P_2&=& N_{2} D_{2}^{-1}= \widetilde{D}_{2}^{-1} \widetilde{N}_{2},
\end{eqnarray*}
we define
\begin{equation}
\label{eq_nu_metric}
d_{\nu} (P_1,P_2 ):=\left\{
\begin{array}{ll}
  \|\widetilde{G}_{2} G_{1}\| &
  \textrm{if } \det(G_1^* G_2) \in \inv S \textrm{ and }
  \iota (\det (G_1^* G_2))=\circ, \\
  1 & \textrm{otherwise}. \end{array}
\right.
\end{equation}
\end{definition}

The following was proved in \cite{BalSas}:

\begin{theorem}
\label{thm_d_nu_is_a_metric}
$d_\nu$ given by \eqref{eq_nu_metric} is a metric on $\mS(R, p, m)$.
\end{theorem}

\begin{definition}
   Given $P \in (\mF(R))^{p\times m}$ and $C\in (\mF(R))^{m\times p}$,
   the {\em stability margin} of the pair $(P,C)$ is defined by
$$
\mu_{P,C}=\left\{ \begin{array}{ll}
\|H(P,C)\|_{\infty}^{-1} &\textrm{if }P \textrm{ is stabilized by }C,\\
0 & \textrm{otherwise.}
\end{array}\right.
$$
\end{definition}

The number $\mu_{P,C}$ can be interpreted as a measure of the
performance of the closed loop system comprising $P$ and $C$: larger
values of $\mu_{P,C}$ correspond to better performance, with
$\mu_{P,C}>0$ if $C$ stabilizes $P$.
 
The following was proved in \cite{BalSas}:

\begin{theorem}
  If $P_0,P\in \mS(R,p,m)$ and $C\in \mS(R, m,p)$, then
$$
  \mu_{P,C} \geq \mu_{P_0,C}-d_{\nu}(P_0,P).
$$
\end{theorem}

The above result says that stabilizability is a robust property of the
plant, since if $C$ stabilizes $P_0$ with a stability margin
$\mu_{P,C}>m$, and $P$ is another plant which is close to $P_0$ in the
sense that $d_\nu(P,P_0)\leq m$, then $C$ is also guaranteed to
stabilize $P$.  

\section{The $\nu$-metric when $R=QA$}
\label{section_Hinfty}

Let $H^\infty$ be the Hardy algebra, consisting of all bounded and
holomorphic functions defined on the open unit disk $\mD:= \{ z\in
\mC: |z| <1\}$. 

As was observed in the Introduction, it was suggested in \cite{Vin} to
use \eqref{eq_nu_metric_GV} to define a metric on the quotient ring of
$H^{\infty}$.  It is tempting to try to do this by using the general
setup of \cite{BalSas} with $R = H^{\infty}$, $S = L^{\infty}$ and
with $\iota$ equal to the Fredholm index of the associated Toeplitz
operator.  However at this level of generality there is no guarantee
that $\varphi$ invertible in $L^{\infty}$ implies that $T_{\varphi}$
is Fredholm (and hence $\iota$ equal to the Fredholm index of the
associated Toeplitz operator is not well-defined on $\textrm{inv }S$
(condition (A3)).  However a perusal of the extensive literature on
Fredholm theory of Toeplitz operators from the 1970s leads to the
choices $R$ equal to the class $QA$ of quasianalytic and $S$ equal to
the class $QC$ of quasicontinuous functions as conceivably the most
general subalgebras of $H^{\infty}$ and $L^{\infty}$ which fit the
setup of \cite{BalSas}, as we now explain.

% Let $H^2$ denote the Hardy space of all holomorphic functions
% $f:\mD\rightarrow \mC$ such that
% $$
% \|f\|_{2}:= \sup_{0<r<1}\left(\frac{1}{2\pi} \int_{0}^{2\pi}
%   |f(re^{i\theta})|^2 d\theta \right)^{\frac{1}{2}}<\infty.
% $$
% Then $H^2$ is a Hilbert space.

The notation $QC$ is used for the $C^*$-subalgebra of $L^\infty(\mT)$
of {\em quasicontinuous} functions:
$$
QC:= (H^\infty + C(\mT) ) \cap \overline{(H^\infty + C(\mT) )}.
$$
An alternative characterization of $QC$ is the following:
$$
QC=L^\infty \cap VMO,
$$
where $VMO$ is the class of functions of vanishing mean oscillation
\cite[Theorem~2.3, p.368]{Gar}. 

The Banach algebra $QA$ of analytic quasicontinuous functions is
$$
QA:= H^\infty \cap QC.
$$
We have the following. 

In order to verify (A4), we will also use the result given below; see
\cite[Theorem~7.36]{Dou}.

\begin{proposition}
\label{prop_Dou}
If $f\in H^\infty(\mD)+C(\mT)$, then $T_f$ is Fredholm if and only if there exist
$\delta, \epsilon>0$ such that
$$
|F(re^{it})| \geq \epsilon \textrm{ for } 1-\delta <r<1,
$$
where $F$ is the harmonic extension of $f$ to $\mD$. Moreover, in this
case the index of $T_f$ is the negative of the winding number with
respect to the origin of the curve $F(re^{it})$ for $1-\delta <r<1$.
\end{proposition}

\begin{theorem}
\label{lemma_disk_algebra}
  Let
\begin{eqnarray*}
  R&:=& QA, \\
  S&:=& QC, \phantom{\Big(} \\
  G&:=& \mZ, \\
  \iota&:=&\Big(\varphi (\in \textrm{\em inv } QC) \mapsto \textrm{\em Fredholm index of } T_{\varphi}(\in \mZ)\Big).
\end{eqnarray*}
Then {\em (A1)-(A4)} are satisfied.
\end{theorem}
\begin{proof} Since $QA$ is a commutative integral domain with
  identity, (A1) holds. 

  The set $QC$ is a unital ($1\in C(\mT)\subset QC$), commutative,
  complex, semisimple Banach algebra with the
  involution
$$
f^*(\zeta)=\overline{f(\zeta)} \quad (\zeta \in \mT).
$$
In fact, $QC$ is a $C^*$-subalgebra of $L^\infty(\mT)$. So (A2) holds
as well. 

\cite[Corollary~139, p.354]{Nik} says that if $\varphi\in \inv QC$,
then $T_\varphi$ is a Fredholm operator.  Thus it follows that the map
$\iota: \inv QC \rightarrow \mZ$ given by
$$
\iota (\varphi):=\textrm{Fredholm index of } T_{\varphi} \quad (\varphi \in \inv QC)
$$
is well-defined. If $\varphi, \psi \in \inv QC$, then in particular
they are elements of $H^\infty+C(\mT)$, and so the semicommutator 
$$
T_{\phi \psi}-T_{\phi} T_{\psi}
$$
is compact \cite[Lemma~133, p.350]{Nik}.  Since the Fredholm index is
invariant under compact perturbations (see e.g. \cite[Part B,
2.5.2(h)]{Nik}), it follows that the Fredholm index of $T_{\varphi
  \psi}$ is the same as that of $T_{\phi}T_{\psi}$. Consequently
(A3)(I1) holds.

Also, if $\varphi \in \inv QC$, then we have that 
\begin{eqnarray*}
\iota (\varphi^*)&=&\iota (\overline{\varphi})\\
&=& \textrm{Fredholm index of }T_{ \overline{\varphi}}\\
&=&\textrm{Fredholm index of }(T_{ \varphi})^*\\
&=&-(\textrm{Fredholm index of }T_{ \varphi})\\
&=&-\iota (\varphi).
\end{eqnarray*}
Hence (A3)(I2) holds.

The map sending the a Fredholm operator on a Hilbert space to its
Fredholm index is locally constant; see for example \cite[Part B,
2.5.1.(g)]{NikE}. For $\varphi \in L^\infty(\mT)$, $\|T_\varphi\|\leq
\|\varphi\|$, and so the map $\varphi\mapsto T_{\varphi}:\inv QC
\rightarrow \textrm{Fred}(H^2)$ is continuous. Consequently the map
$\iota$ is continuous from $\inv QC$ to $\mZ$ (where $\mZ$ has the
discrete topology). Thus (A3)(I3) holds.

Finally, we will show that (A4) holds as well. Let $\varphi \in
H^\infty \cap (\inv QC)$ be invertible as an element of $H^\infty $.
Then clearly $T_\varphi$ is invertible, and so has Fredholm index ind
$T_{\varphi}$ equal to $0$. Hence $\iota (\varphi)=0$. This finishes
the proof of the ``only if'' part in (A4).

Now suppose that $\varphi \in H^\infty \cap (\inv QC)$ and that $\iota
(\varphi)=0$. In particular, $\varphi$ is invertible as an element of $H^\infty +
C({\mathbb T})$ and the Fredholm index ind $T_{\varphi}$ of $T_{\varphi}$
is equal to 0. 
%In particular, $\varphi$ is invertible as an element of
%$H^\infty+C(\mT)$ and that the Fredholm index of $\ind T_\varphi=0$.
By Proposition~\ref{prop_Dou}, it follows that there exist $\delta,
\epsilon>0$ such that $|\Phi(re^{it})| \geq \epsilon$ for $ 1-\delta
<r<1$, where $\Phi$ is the harmonic extension of $\varphi$ to $\mD$.
But since $\varphi \in H^\infty$, its harmonic extension $\Phi$ is
equal to $\varphi$. So $|\varphi(re^{it})| \geq \epsilon$ for $
1-\delta <r<1$. Also since $\iota(\varphi)=0$, the winding number with
respect to the origin of the curve $\varphi(re^{it})$ for $1-\delta
<r<1$ is equal to $0$. By the Argument principle, it follows that $f$
cannot have any zeros inside $r\mT$ for $1-\delta <r<1$. In light of
the above, we can now conclude that there is an $\epsilon'>0$ such
that $|\varphi(z)|>\epsilon'$ for all $z\in \mD$.  Thus
$1/\varphi$ is in $H^{\infty}$ with $H^{\infty}$-norm at most $1/
\epsilon'$ and we conclude that $\varphi$ is invertible as an element
of $H^{\infty}$.  Consequently (A4) holds.
\end{proof}

\noindent In the definition of the $\nu$-metric given in
Definition~\ref{def_nu_metric} corresponding to
Lemma~\ref{lemma_disk_algebra}, the $\|\cdot\|_\infty$ now means the
  usual $L^\infty(\mT)$ norm.

\begin{lemma}
Let $A\in QC^{p\times m}$. Then 
$$
\|A\|=\|A\|_\infty:= \textrm{\em ess.sup}_{\zeta \in\mT}\nm A(\zeta)\nm.
$$
\end{lemma}
\begin{proof} We have that 
\begin{eqnarray*}
  \|A\|_\infty&=& \textrm{ess.sup}_{\zeta \in\mT}\nm A(\zeta)\nm=
  \textrm{ess.sup}_{\zeta \in\mT} \sigma_{\scriptscriptstyle \textrm{max}} \Big(A(\zeta)\Big)\phantom{\displaystyle{\sup_{\in \mT}}}\\
  &=& %\|(\mT \owns )\zeta \mapsto  \sigma_{\scriptscriptstyle \textrm{max}} (A(\zeta))\|_{L^\infty(\mT)}=
  \max_{\varphi \in M(L^\infty(\mT))} \widehat{\sigma_{\scriptscriptstyle \textrm{max}} (A)}(\varphi)=
  \max_{\varphi \in M(L^\infty(\mT))} \sigma_{\scriptscriptstyle \textrm{max}}\Big(\widehat{A}(\varphi)\Big)\\
  &=& \max_{\varphi \in M(QC)} 
  \sigma_{\scriptscriptstyle \textrm{max}}\Big(\widehat{A}(\varphi)\Big)= \max_{\varphi \in M(QC)} \nm \widehat{A}(\varphi)\nm= \|A\|.
\end{eqnarray*}
In the above, the notation $\sigma_{\max}(X)$, for a complex matrix
$X\in \mC^{p\times m}$, means its largest singular value, that is, the
square root of the largest eigenvalue of $X^*X$ (or $XX^*$). We have
also used the fact that for an $f\in QC \subset
L^\infty(\mT)$, we have that 
$$
\max_{\varphi \in M(L^\infty(\mT))} \widehat{f}(\varphi)=\|f\|_{L^\infty(\mT)}=
 \max_{\varphi \in M(QC)} \widehat{f}(\varphi).
$$
Also, we have used the fact that if $\mu\in L^\infty(\mT)$ is such that 
$$
\det(\mu^2 I-A^*A)=0,
$$
then upon taking Gelfand transforms, we obtain
$$
\det((\widehat{\mu}(\varphi))^2
I-(\widehat{A}(\varphi))^*\widehat{A}(\varphi))=0\quad (\varphi\in
M(L^\infty(\mT))),
$$
to see that $\widehat{\sigma_{\scriptscriptstyle \textrm{max}}
  (A)}(\varphi)=\sigma_{\scriptscriptstyle
  \textrm{max}}(\widehat{A}(\varphi))$, $\varphi\in M(L^\infty(\mT))$.
\end{proof}

Finally, our scalar winding number condition 
$$
\det(G_1^* G_2) \in \inv QC \textrm{ and Fredholm index of }
  T_{\det (G_1^* G_2))}=0
$$
is exactly the same as the condition 
$$
T_{G_1^* G_2} \textrm{ is Fredholm with Fredholm index } 0
$$
in \eqref{eq_nu_metric_GV}. This is an immediate consequence of the
following result due to Douglas~\cite[p.13, Theorem 6]{Dou2}.

\begin{proposition}
  The matrix Toeplitz operator $T_\Phi$ with the matrix symbol
  $\Phi=[\varphi_{ij}]\in (H^\infty+C(\mT))^{n\times n}$ is Fredholm
  if and only if
$$
\inf_{\zeta \in \mT} |\det ( \varphi(\zeta))|>0,
$$
and moreover the Fredholm index of $T_\Phi$ is the negative of the
Fredholm index of $\det \Phi$.
\end{proposition}

Thus our abstract metric reduces to the same metric given in
\eqref{eq_nu_metric_GV}, that is, for plants $P_1,P_2\in \mS(QA,p,m)$,
with the normalized left/right coprime factorizations
\begin{eqnarray*}
P_1&=& N_{1} D_{1}^{-1}= \widetilde{D}_{1}^{-1} \widetilde{N}_{1},\\
P_2&=& N_{2} D_{2}^{-1}= \widetilde{D}_{2}^{-1} \widetilde{N}_{2},
\end{eqnarray*}
define
\begin{equation}
\label{eq_nu_metric_aaa}
d_{\nu} (P_1,P_2 ):=\left\{
\begin{array}{ll}
  \|\widetilde{G}_{2} G_{1}\|_{\infty} &
  \textrm{if } \det(G_1^* G_2) \in \inv QC \textrm{ and }\\
& \quad \textrm{ Fredholm index of }
  T_{\det (G_1^* G_2)}=0, \\
  1 & \textrm{otherwise}. \end{array}
\right.
\end{equation}

Summarizing, our main result is the following.

\begin{corollary}
  $d_\nu$ given by \eqref{eq_nu_metric_aaa} is a metric on
  $\mS(QA,p,m)$. Moreover, if $P_0,P\in \mS(QA,p,m)$ and $C\in \mS(QA,
  m,p)$, then
$$
  \mu_{P,C} \geq \mu_{P_0,C}-d_{\nu}(P_0,P).
$$
\end{corollary}

\end{document}